\documentclass[a4paper]{amsart}
\usepackage{mathtools,amssymb,upref,url,xcolor,etoolbox}
\definecolor{darkred}{rgb}{0.6,0,0}
\definecolor{darkgreen}{rgb}{0,0.5,0}
\definecolor{darkmagenta}{rgb}{0.5,0,0.5}
\usepackage[pdftex, colorlinks=true,
  linkcolor=darkred,
  citecolor=darkgreen,
  urlcolor=darkmagenta]{hyperref}

\newtheorem{theorem}{Theorem}
\newtheorem{lemma}[theorem]{Lemma}

\allowdisplaybreaks

\swapnumbers
\newcommand{\cond}[1]{\textup{(\textit{#1})}}

\DeclareMathOperator{\supp}{supp}
\DeclareMathOperator{\dist}{dist}
\newcommand{\NN}{\mathbb{N}}
\newcommand{\RR}{\mathbb{R}}
\newcommand{\sF}{\mathcal{F}}
\newcommand{\charfun}[1]{\mathbf{1}_{#1}}
\let\epsilon=\varepsilon

\DeclarePairedDelimiter{\abs}\lvert\rvert
\DeclarePairedDelimiter{\norm}\lVert\rVert
\DeclarePairedDelimiter{\parens}()
\DeclarePairedDelimiter{\braces}\{\}
\newcommand{\setof}[1]{\braces{\,#1\,}}
\newcommand{\given}{\mid}

\newcommand{\dee}{\mathop{}\!{d}} 

\begin{document}
\title[Kolmogorov--Riesz compactness theorem]{An improvement of the Kolmogorov--Riesz compactness theorem}

\author[Hanche-Olsen]{Harald Hanche-Olsen}
\address[Hanche-Olsen]{\newline
    Department of Mathematical Sciences,
    NTNU Norwegian University of Science and Technology,
    NO--7491 Trondheim, Norway}
\email[]{\href{harald.hanche.olsen@ntnu.no}{harald.hanche.olsen@ntnu.no}}
\urladdr{\href{https://www.ntnu.edu/employees/harald.hanche-olsen}{https://www.ntnu.edu/employees/harald.hanche-olsen}}

\author[Holden]{Helge Holden}
\address[Holden]{\newline
    Department of Mathematical Sciences,
    NTNU Norwegian University of Science and Technology,
    NO--7491 Trondheim, Norway }
\email[]{\href{helge.holden@ntnu.no}{helge.holden@ntnu.no}}
\urladdr{\href{https://www.ntnu.edu/employees/helge.holden}{https://www.ntnu.edu/employees/helge.holden}}

\author[Malinnikova]{Eugenia Malinnikova}
\address[Malinnikova]{\newline
    Department of Mathematical Sciences,
    NTNU Norwegian University of Science and Technology,
    NO--7491 Trondheim, Norway }
\email[]{\href{eugenia.malinnikova@ntnu.no}{eugenia.malinnikova@ntnu.no}}
\urladdr{\href{https://www.ntnu.edu/employees/eugenia.malinnikova}{https://www.ntnu.edu/employees/eugenia.malinnikova}}

\date{\today}



\thanks{Research by HH was supported in part by the grant Waves and Nonlinear Phenomena (WaNP) from the
Research Council of Norway.}


\begin{abstract}
The purpose of this short note is to provide a new and very short proof of a result by Sudakov \cite{Sudakov}, offering an important improvement of the  classical result by Kolmogorov--Riesz on compact subsets of Lebesgue spaces.
\end{abstract}

\errorcontextlines99

\maketitle
\section*{Introduction} \label{sec:intro}

The classical compactness theorem of Kolmogorov--Riesz reads as follows \cite{HOH}:

 A subset $\sF$ of $L^p(\RR^n)$, with  $1\le p<\infty$, is totally bounded if,
 and only if,
  \begin{enumerate}
  \renewcommand{\theenumi}{\alph{enumi}}
  \item $\sF$ is bounded,
  \item for every $\epsilon>0$ there is some $R$ so that, for
  every $f\in\sF$,\[\int_{\abs{x}>R}\abs{f(x)}^p\dee x<\epsilon^p,\]
  \item for every $\epsilon>0$ there is some $\rho>0$ so that, for
  every $f\in\sF$ and $y\in\RR^n$ with $\abs{y}<\rho$,
 \[\int_{\RR^n}\abs{f(x+y)-f(x)}^p\dee x<\epsilon^p.\]
  \end{enumerate}
The purpose of the current paper is to show that the boundedness condition \cond{a}
is redundant.

This was discovered by Sudakov \cite{Sudakov} in 1957,
but the paper appears undeservedly to have been lost in obscurity.
We want to revive the result and present a novel and very short proof
of the redundancy of \cond{a}.

The Kolmogorov--Riesz compactness theorem
was discovered by Kolmogorov \cite{Kolmogorov} in 1931.
He stated the result for a subset of $L^p(\RR^n)$, with $1<p<\infty$,
and the functions in the subset all supported in a common compact set
(thus essentially replacing $\RR^n$ by a bounded subset of $\RR^n$).
Tamarkin \cite{Tamarkin} extended the result to the case of unbounded support
by adding the assumption \cond{b},
and Tulajkov \cite{Tulajkov} extended the result to include $p=1$.
At the same time M.~Riesz \cite{Riesz} proved a similar result.
See \cite{HOH,HOH_add} for a historical account of this result,
various generalizations, and a proof.

The fact that condition \cond{a} is not needed
was only discovered in 1957 by Sudakov \cite{Sudakov}.
The late discovery of this fact
is probably due to a mistake by Tamarkin \cite{Tamarkin},
who presented an erroneous ``example'' in which \cond{b} and \cond{c}
are claimed to be true,
but \cond{a} is false.\footnote
{His example was as follows.
  Consider the family $\sF=\{f_n\}_{n\in\NN}\subset L^p(\RR)$,
  where $f_n(x)=(f(x)+n)\charfun{(0,1)}(x)$ for any $f\in L^p(\RR)$.
  Clearly, $\sF$ satisfies \cond{b},
  but neither condition \cond{a} nor \cond{c},
  and $\sF$ is not totally bounded.}
Sudakov \cite{Sudakov} states that Tamarkin's mistake was discovered by Natanson,
but gives no reference.
The result by Sudakov has recently been revisited
in the context of  metric measure spaces in \cite{GorkaRafeiro}.
See also \cite{Brudnyi,Rafeiro}. These proofs all rely intrinsically on the
approach by Sudakov, but are applied to more general spaces.

The Kolmogorov--Riesz compactness theorem is really a classical textbook result,
and it is always stated as giving necessary and sufficient conditions
for a subset of a Lebesgue space to be compact.
The fact that one condition is not needed should be more widely known,
and this is our reason for publishing this result.

\section*{The improved Kolmogorov--Riesz--Sudakov compactness result}

Thanks to Sudakov's discovery,
the original Kolmogorov--Riesz theorem
admits the following improvement:

\begin{theorem}[Kolmogorov--Riesz--Sudakov] \label{thm:krs}
  Let $1\le p<\infty$.
  A subset $\sF$ of $L^p(\RR^n)$ is totally bounded if,
  and only if,
  \begin{enumerate}
  \renewcommand{\theenumi}{\roman{enumi}}
  \item for every $\epsilon>0$ there is some $R$ so that, for
  every $f\in\sF$,\[\int_{\abs{x}>R}\abs{f(x)}^p\dee x<\epsilon^p,\]
  \item for every $\epsilon>0$ there is some $\rho>0$ so that, for
  every $f\in\sF$ and $y\in\RR^n$ with $\abs{y}<\rho$,
 \[\int_{\RR^n}\abs{f(x+y)-f(x)}^p\dee x<\epsilon^p.\]
  \end{enumerate}
\end{theorem}

\medskip\noindent\textit{Remark.}
  Observe that in the case where $\sF$ is a subset of $L^p(\Omega)$,
  where $\Omega$ is a bounded subset  of $\RR^n$,
  only the condition of ``$L^p$ equicontinuity'',
  that is, condition \cond{ii},
  is necessary and sufficient for  $\sF$  to be totally bounded.
  However, this condition must be interpreted with care,
  by identifying $L^p(\Omega)$ with a subspace of  $L^p(\RR^n)$.
  Thus the behavior of functions in $\sF$ at the boundary of $\Omega$
  will influence whether \cond{ii} holds or not.
  This can be illustrated by the failure of Tamarkin's example;
  see the footnote in the introduction.

\medskip

Before embarking on the proof,
we establish some notation.
Throughout, $B_r(x)$ denotes the open ball of radius $r$ centered at $x\in\RR^n$.
We sometimes write $B_r$ instead of $B_r(0)$.
We write $\charfun{A}$ for the characteristic function
of a set $A\subseteq\RR^n$.
The translation operator $T_y$ is defined by $T_yf(x)=f(x-y)$.
When $\Omega\subseteq\RR^n$,
we identify $L^p(\Omega)$ with the set of functions in $L^p(\RR^n)$
vanishing outside $\Omega$.
We write $X_1$ for the closed unit ball of any normed space $X$.

In light of the classical Kolmogorov--Riesz theorem, see, e.g., \cite{HOH},
the following is all that is required to prove  Theorem \ref{thm:krs}:

\begin{proof}[Proof that \cond{b} and \cond{c} imply \cond{a}]
  Assume that conditions \cond{b} and \cond{c} are satisfied.
  Due to condition \cond{b} we only need to bound
  the norm uniformly on some sufficiently large ball.
  The idea is that by \cond{c},
  small translation are uniformly close to the identity in the $L^p(\RR^n)$ norm.
  By restricting to a ball, and repeating the small translation,
  we can get an estimate of the norm on a ball
  by the norm on a translated ball
  that is contained in the domain of integration in \cond{b},
  which gives the uniform bound we want.

  More precisely, fix $\epsilon=1$ and let
  $R>0 $ and $\rho>0$ be the corresponding quantities
  given by \cond{b} and \cond{c}.
  For any $f\in\sF$,
  using the triangle inequality and a translation, we infer%
 \begin{align*}
   \norm{f\charfun{B_R(z)}}_p
   &\le \norm{(T_yf-f)\charfun{B_R(z)}}_p +\norm{f\charfun{B_R(z-y)}}_p \\
   &\le \norm{(T_yf-f)}_p +\norm{f\charfun{B_R(z-y)}}_p\\
   &\le 1+\norm{f\charfun{B_R(z-y)}}_p.
\end{align*}
Here $y\in\RR^n$ is any nonzero vector with $\abs{y}<\rho$.
By induction, we find that
 \[
   \norm{f\charfun{B_R(0)}}_p\le  N+ \norm{f\charfun{B_R(-Ny)}}_p.
 \]
Choosing $N$ so that $N\abs{y}> 2R$,
we see that $B_{R}(-Ny)\cap B_{R}(0)=\emptyset$, and
 \[
   \norm{f}_p
   = \norm{f\charfun{B_{R}(0)}}_p+  \norm{f\charfun{\RR^n\setminus B_{R}(0)}}_p
   \le  N+ 2,
 \]
uniformly in $f$.

\end{proof}

Sudakov states the theorem with the translate $T_yf$
in \cond{ii} replaced by the \emph{Steklov mean}
\[
  S_hf(x)=\abs{B_h}^{-1}\int_{B_h}f(x+y) \dee y=\abs{B_h}^{-1} f*\charfun{B_h}(x)
\]
for sufficiently small $h$,
where $\abs{B_h}$ denotes the volume of $B_h$.
Clearly, the revised condition follows from \cond{ii},
but the converse is far from obvious.
We  show that Sudakov's condition can also be used
instead of  \cond{ii} to estimate the $L^p$-norm:

\begin{theorem}[Kolmogorov--Riesz--Sudakov] \label{thm:krsMEAN}
  Theorem \ref{thm:krs} holds with condition \cond{ii} replaced by\textup{:}
 \begin{enumerate}
  \item[\cond{ii'}]
    For every $\epsilon>0$ there is some $\rho>0$ so that, for
    every $f\in\sF$ and $h$ with $0<h<\rho$,
    \[\int_{\RR^n}\abs{f(x)-S_hf(x)}^p\dee x<\epsilon^p.\]
  \end{enumerate}
\end{theorem}

We will need a lemma.
\begin{lemma} \label{lm:cc}
  Assume that
  $p$ and $q$ are conjugate exponents with $1\le p < \infty$, and that
  $\phi\in L^q(\RR^n)$ has compact support.
  If $p=1$, assume further that $\phi$ is continuous.
  Let $K\subset\RR^n$ be compact.
  Then the map $\Phi\colon L^p(K)\to L^p(\RR^n)$ defined by $\Phi f=\phi*f$
  is compact.
\end{lemma}

\begin{proof}
  First note that $y\mapsto T_y\phi$ is a continuous map $\RR^n\to L^q(\RR^n)$
  (see, e.g., \cite[Prop. 20.1]{DiBenedetto}).
  It immediately follows that the set of functions
  $\setof{\phi*f\given f\in L^p(\RR^n)_1}$
  is equicontinuous, since
  \[
    \begin{split}
      \abs{\phi*f(x-y)-\phi*f(x)}
      &=\abs{(T_y\phi-\phi)*f(x)}\\&
      \le\norm{T_y\phi-\phi}_q\cdot\norm{f}_p
      \le\norm{T_y\phi-\phi}_q
    \end{split}
  \]
  for any $f\in L^p(\RR^n)_1$
  (the continuity of $\phi$ when $p=1$
  is required in order to make $\norm{T_y\phi-\phi}_q$ small even in that case).
  A similar estimate shows that this set of functions
  is uniformly bounded.
  Since all functions $\phi*f$ with $f\in L^p(K)$
  are supported by the compact set $K+\supp \phi$,
  we can now employ the Arzel\`a--Ascoli theorem
  to conclude that
  $\setof{\phi*f\given f\in L^p(K)_1}$
  is totally bounded in the uniform norm.
  Again, because of the shared compact support,
  this implies compactness in $L^p(\RR^n)$.
\end{proof}

\begin{proof}[Proof of Theorem \ref{thm:krsMEAN}]
  \label{proof:krsMEAN}
  In the proof of Theorem \ref{thm:krs},
  we employed repeated translations to move the support of $T_{-Ny}f$
  outside a large ball.
  Here we use instead repeated applications of the convolution operator  (the Steklov mean) $S_h$
  for a similar purpose, getting
  some weighted average of $f$.
We cannot move the whole weight to the complement of some fixed ball as before, however.
Instead, we notice that the total weight is one,
but some fixed part of it is moved to this complement.

To make this presice,
we start by fixing $R$ as given by \cond{i}
and $\rho$ as given by \cond{ii'},
both with $\epsilon=1$.
Let $0<h<\rho$, and put $\phi=\abs{B_h}^{-1}\charfun{B_h}$.
Select a natural number $N$ so that $Nh>2R$, and put
\[
  \psi=\phi^{*N}=\underbrace{\phi*\cdots*\phi}_{\text{$N$ times}}.
\]
Writing the $N$-fold convolution as an ($N-1$)-fold
integral over $z_1$, \ldots, $z_{N-1}\in\RR^{n}$
and setting $z_N=x-z_1-\cdots-z_{N-1}$, we can write this as
\[
  \psi(x)=\mathop{\int\cdots\int}\limits_{z_1+\cdots+z_N=x}
  \phi(z_1)\cdots\phi(z_N)\dee z_1\cdots\dee z_{N-1},
\]
from which
it follows that $\psi(x)>0$ when $\abs{x}<Nh$,
and $\psi(x)=0$ otherwise.
Note also that $\int_{\RR^n}\psi\dee x=1$.

Now fix some $f\in\sF$, and define
\[
  A(y)=\norm{f\charfun{B_R(y)}}_p
  =\parens[\Big]{\int_{B_R}\abs{f(x+y)}^p\dee x}^{1/p}.
\]
Our task is to find a bound for $A(0)$, independent of $f$.
Together with \cond{i}, this will establish
a uniform bound on $\norm{f}_p$ for $f\in\sF$.

The function $A$ is continuous,
since we can also write
$A(y)=\norm{(T_yf)\charfun{B_R(0)}}_p$.
Further, condition \cond{i} implies that $A(y)<1$
for $\abs{y}\ge2R$,
so $A$ is certainly bounded.
Let $M=\sup_{y\in\RR^n}A(y)$.

To estimate $A(y)$, we break it up as follows:
\begin{equation}\label{eq:A1}
  A(y) \le \norm{(f*\psi)\charfun{B_R(y)}}_p + \norm{(f*\psi-f)\charfun{B_R(y)}}_p.
\end{equation}
For the first term,
the continuous Minkowski inequality
(see, e.g., \cite[Prop.~4.3 (p.~227)]{DiBenedetto}) yields
\[
  \begin{split}
    \norm{(f*\psi)\charfun{B_R(y)}}_p
    &=\parens[\bigg]{
      \int_{B_R(y)}\abs[\Big]{\int_{\RR^n} f(x-u)\psi(u)\dee u}^p\dee x}^{1/p}\\
    &\le \int_{\RR^n}\parens[\bigg]{\int_{B_R(y)}\abs{f(x-u)^p}\dee x}^{1/p}\psi(u)\dee u\\
    &= A*\psi(y).
  \end{split}
\]
As for the second term of \eqref{eq:A1},
first note that condition \cond{ii'} with $\epsilon=1$ can be written
 $\norm{f*\phi-f}_p<1$.
Furthermore, $\norm{g*\phi}_p\le\norm{g}_p$ for any $g \in L^p(\RR^n)$
(as seen, e.g., by another application of the continuous Minkowski inequality).
Thus we find
$\norm{f*\phi^{*(k+1)}-f*\phi^{*k}}_p
\le\norm{(f*\phi-f)*\phi^{*k}}_p\le\norm{f*\phi-f}_p<1$,
so  we have
\begin{equation} \label{eq:rconvnorm}
  \norm{f*\phi^{*k}-f}_p \le k \qquad (f \in \sF).
\end{equation}
In particular, $\norm{f*\psi-f}_p \le N$,
and so \eqref{eq:A1} reduces to
\[
  A \le A*\psi+N.
\]
However,
\[
  \begin{split}
    A*\psi(y)
    &=\int_{\RR^n} A(u)\psi(y-u)\dee u\\
    &\le M\int_{B_{2R}}\psi(y-u)\dee u+\int_{\RR \setminus B_{2R}}\psi(y-u)\dee u\\
    &\le M\gamma+1,
  \end{split}
\]
where
\[
  \gamma=\max_{y\in\RR^n}\int_{B_{2R}}\psi(y-u)\dee u<1.
\]
Indeed, note that the above integral is a continuous function of $y$,
with compact support, so it achieves its maximum.
But the integral is always less than $1$,
because the integrand is strictly positive in a ball of radius $Nh>2R$.

To summarize, we have $M=\sup_yA(y) \le M\gamma+1+N$,
and therefore $M \le (1+N)/(1-\gamma)$.
Since this estimate is independent of $f$,
we have now proved that $\sF$ is bounded in $L^p(\RR^n)$.

To finish the proof,
let $\epsilon>0$, once more pick $R>0$ and $\rho>0$
according to conditions \cond{i} and \cond{ii'},
and let $\phi=\abs{B_h}^{-1}\charfun{B_h}$,
where $0 < h < \rho$.
Define the linear map $\Phi_R\colon\sF \to L^p(\RR^n)$ by
\[
  \Phi_R f = (f\charfun{B_R})*\phi*\phi
\]
(we may replace $\phi*\phi$ by $\phi$, if $p \ne 1$).
It is compact, by Lemma \ref{lm:cc}.
Therefore, since $\sF$ is bounded,
$\Phi_R\sF$ is totally bounded.
Now, for any $f\in\sF$,
\[
  \norm{f-\Phi_R f}_p
  \le \norm{f-f*\phi*\phi}_p+\norm{(f-f\charfun{B_R})*\phi*\phi}_p
  <2\epsilon+\epsilon=3\epsilon.
\]
Here the first norm estimate comes from \eqref{eq:rconvnorm},
while the second one is due to \cond{i}
and the general fact that $\norm{g*\phi}_p \le \norm{g}_p$.

Thus any member of $\sF$ is within a distance $3\epsilon$
of some member of the totally bounded set $\Phi_R\sF$,
and so $\sF$ itself is totally bounded.
\end{proof}

\section*{Review of the original proof of Sudakov}
For the benefit of the reader we review Sudakov's original argument,
which is interesting for two reasons.
First of all it is quite different from other proofs of this theorem,
and, furthermore,
it uses only conditions \cond{i} and \cond{ii'}
without involving the uniform boundedness.
We start by stating and proving two general results.

\begin{theorem}[{Mazur, see \cite[p.~466]{Banach}}] \label{tm:mazur}
  Let $G$ be a bounded subset of a Banach space $X$.
  Assume that $(U_k)$ is a sequence of compact operators on $X$
  converging to the identity operator in the strong operator topology,
  i.e., $\norm{U_kx-x}\to0$ for all $x\in X$.
  Then $G$ is totally bounded if, and only if,
  $\norm{U_kx-x}\to0$ uniformly for $x\in G$.
\end{theorem}

\begin{proof}
  First, assume that $\norm{U_kx-x}\to0$ uniformly for $x\in G$.
  Then for any $\epsilon>0$, there is some $k$
  so that $\dist(x,U_kG)<\epsilon$ for all $x\in G$.
  The image $U_kG$ is totally bounded,
  because $G$ is bounded and $U_k$ is compact.
  The total boundedness of $G$ follows.

  Conversely, assume $G$ is totally bounded.
  Apply the Banach--Steinhaus theorem
  to get a uniform bound $\norm{U_k}\le M$ for all $k$.
  If $\epsilon>0$, there is an $\epsilon$-net $F\subseteq G$:
  A~finite set so that every point in $G$
  is within a distance $\epsilon$ from some member of~$F$.
  If $k$ is large enough, $\norm{U_ky-y}\le\epsilon$ for all $y\in F$.
  For any $x\in G$, then, there is some $y\in F$ with $\norm{y-x}<\epsilon$,
  and so
  \[
    \norm{U_kx-x}\le\norm{U_k(x-y)}+\norm{U_ky-y}+\norm{y-x}
    <M\epsilon+\epsilon+\epsilon=(M+2)\epsilon.
  \]
  Since $M$ is fixed and $\epsilon$ is arbitrary,
  $\norm{U_kx-x}\to0$ uniformly for $x\in G$.
\end{proof}

\begin{lemma}[Sudakov \cite{Sudakov}] \label{lm:sudakov}
  Assume that $X$ is a Banach space, and $G\subseteq X$.
  Assume also that $U$ is a compact operator on $X$
  so that $1$ is not an eigenvalue of $U$,
  and $\norm{Ux-x}\le M<\infty$ for all $x\in G$.
  Then $G$ is bounded.
\end{lemma}
\begin{proof}
  Since $U$ is compact and $1$ is not an eigenvalue, $1\notin\sigma(U)$,
  and so $U-I$ is invertible.
  So for any $x\in G$, $\norm{x}\le{\norm{(U-I)^{-1}}}\cdot{\norm{Ux-x}}\le{\norm{(U-I)^{-1}}}M$.
\end{proof}

\begin{proof}[A different proof of Theorem \ref{thm:krsMEAN}]
  We prove only that \cond{i} and \cond{ii'}
  imply total boundedness.
  For the other direction, refer to the earlier proof
  (see page~\pageref{proof:krsMEAN}).

  For any $\epsilon>1$, choose $R$ according to condition \cond{i},
  and define a continuous cutoff function $v_{R}$:
  \[
    v_{R}(x)=
    \begin{cases}
      1 & \abs{x}<R+1,\\
      R+2-\abs{x} & R+1\le\abs{x}\le R+2,\\
      0 & \abs{x}>R+2.
    \end{cases}
  \]
  Thus $\norm{f-fv_{R}}_p<\epsilon$ for any $f\in\sF$.
  If we can show that $\sF v_{R}$ is totally bounded for every $R>0$,
  it immediately follows that $\sF$ is totally bounded.

  We now observe that condition \cond{ii'} is still satisfied
  if $\sF$ is replaced by $\sF v_{R}$.
  To see this, note that
  \[
    \begin{split}
      \norm{fv_{R}-S_h(fv_{R})}_p
      &\le\norm{(f-S_hf)v_{R}}_p+\norm{(S_hf)v_{R}-S_h(fv_{R})}_p\\
      &\le\norm{f-S_hf}_p+\norm{(S_hf)v_{R}-S_h(fv_{R})}_p.
    \end{split}
  \]
  Next,
  \[
    S_hf(x)v_{R}(x)-S_h(fv_{R})(x)
    =\abs{B_h(x)}^{-1}\int_{B_h(x)}f(y)\parens[\big]{{v_{R}(x)-v_{R}(y)}}\dee y.
  \]
  Note that $\abs{v_{R}(x)-v_{R}(y)}\le\abs{x-y}<h$ whenever $y\in B_h(x)$,
  and furthermore $v_{R}(x)-v_{R}(y)=0$ if in addition $\abs{x}\le R$,
  provided we ensure that $h<1$.
  Under this assumption, then,
  \[
    \abs{S_hf(x)v_{R}(x)-S_h(fv_{R})} \le hS_h\abs{f\charfun{\RR \setminus B_R}}(x),
  \]
  and therefore
  \[
    \norm{(S_hf)v_{R}-S_h(fv_{R})}_p
    \le h\norm{S_h\abs{f\charfun{\RR \setminus B_R}}}_p
    \le h\norm{f\charfun{\RR \setminus B_R}}_p
    < h.
  \]
  And so we get
  \[
    \norm{fv_{R}-S_h(fv_{R})}_p \le\norm{f-(S_hf)}_p+h,
  \]
  and it follows that $\sF v_{R}$ does indeed satisfy \cond{ii'}.
  Thus we can replace $\sF$ with $\sF v_{R}$
  in the remainder of the proof.

  \medskip\noindent
  From now on, we assume without loss of generality
  that $\supp{f}\subseteq K$ for all $f\in\sF$,
  where $K\subset\RR$ is compact.
  Let $\phi_k=\abs{B_{1/k}}^{-1}\charfun{B_{1/k}}$.
  Then $f*\phi_k=S_{1/k}f \to f$ in the $L^p$ norm, uniformly for $f\in\sF$;
  and the same is true for $f*\phi_k*\phi_k$.

  Define the operator $\Phi_k\colon L^p(K) \to L^p(K)$ by
  $\Phi_kf=(f*\phi_k*\phi_k)\charfun{K}$.
  Lemma \ref{lm:cc} ensures that $\Phi_k$ is compact.

  We claim that $1$ is not an eigenvalue of $\Phi_k$.
  Assuming this, we can use Lemma~\ref{lm:sudakov}
  to conclude that $\sF$ is bounded,
  and then Mazur's theorem (Theorem \ref{tm:mazur})
  implies that $\sF$ is totally bounded,
  thus finishing the proof.

  To prove the claim, assume the contrary,
  and let a nonzero $f \in L^p(K)$
  satisfy $f=(f*\psi)\charfun{K}$, where $\psi=\phi_k*\phi_k$.
  Without loss of generality, we may assume that $f(x)>0$ for some $x$.
  Note that $f*\psi$ is continuous,
  and so $f$ has a maximum value $c>0$.
  Let $C \subseteq K$ be the compact set $\setof{x\in\RR^n\given f=c}$,
  and consider any point $x$ on the boundary of $C$.
  Then we have
  \[
    c=f(x)=\int_{\RR^n}f(x-y)\psi(y)\dee y.
  \]
  Since $f \le c$, and $f(x-y)<c$ for $y$ in some open set in which $\psi(y)>0$,
  we get
  \[
    \int_{\RR^n}f(x-y)\psi(y)\dee y<c\int_{\RR^n}\psi(y)\dee y=c,
  \]
  and so we arrive at the contradiction $c<c$.
  This completes the proof.
\end{proof}



\end{document}